\documentclass[11pt]{amsart}
\usepackage{amsfonts, amssymb, amsmath}
\usepackage{tikz-cd}
\usepackage{url}
\usepackage{bbm}
\usepackage[hidelinks]{hyperref}
\usepackage[utf8]{inputenc}

\usepackage{url}

\tikzset{
	labl/.style={anchor=south, rotate=90, inner sep=.5mm}
}
\newcommand{\Q}{\mathbb{Q}}
\newcommand{\C}{\mathbb{C}}
\newcommand{\F}{\mathbb{F}}
\newcommand{\Ga}{\mathbb{G}_a}

\newcommand{\KK}{\mathbb{K}}
\newcommand{\Z}{\mathbb{Z}}
\newcommand{\N}{\mathbb{N}}

\newcommand{\Crys}{\mathbf{Crys}}
\newcommand{\bfC}{\mathbf{C}}

\newcommand{\Foi}{\Fisoc^\mathrm{\dag}}
\newcommand{\oi}{\Isoc^\mathrm{\dag}}

\newcommand{\iso}{\xrightarrow{\sim}}

\newcommand{\Zp}{\mathbb{Z}_p}

\newcommand{\Qp}{\mathbb{Q}_p}

\newcommand{\Qlbar}{\overline{\mathbb{Q}}_{\ell}}
\newcommand{\Qpbar}{\overline{\mathbb{Q}}_{p}}

\newcommand{\et}{\mathrm{\acute{e}t}}
\newcommand{\tors}{\mathrm{tors}}
\newcommand{\crys}{\mathrm{crys}}

\newcommand{\perf}{\mathrm{perf}}
\newcommand{\Fq}{\F_q}

\newcommand{\frakA}{\mathfrak{A}}

\newcommand{\calE}{\mathcal{E}}
\newcommand{\calF}{\mathcal{F}}
\newcommand{\calG}{\mathcal{G}}

\newcommand{\calO}{\mathcal{O}}

\newcommand{\calT}{\mathcal{T}}

\newcommand{\pure}{\mathrm{pure}}
\newcommand{\pured}{\mathrm{pure}^{\dagger}}

\newcommand{\Spec}{\mathrm{Spec}}

\newcommand{\Hom}{\mathrm{Hom}}

\newcommand{\Ext}{\mathrm{Ext}}

\newcommand{\id}{\mathrm{id}}

\newcommand{\Tr}{\mathrm{Tr}}

\newcommand{\Fcrys}{\mathbf{F\textrm{-}Crys}}
\newcommand{\cst}{\mathrm{cst}}
\newcommand{\rk}{\mathrm{rk}}

\newcommand{\Fisoc}{\mathbf{F\textrm{-}Isoc}}
\newcommand{\Isoc}{\mathbf{Isoc}}

\newcommand{\etale}{\'etale }

\begin{document}

	\newtheorem{theo}[subsubsection]{Theorem}
	\newtheorem*{theo*}{Theorem}
	\newtheorem{ques}[subsubsection]{Question}
	\newtheorem*{ques*}{Question}
	\newtheorem{conj}[subsubsection]{Conjecture}
	\newtheorem{prop}[subsubsection]{Proposition}
	\newtheorem{lemm}[subsubsection]{Lemma}
	\newtheorem*{lemm*}{Lemma}
	\newtheorem{coro}[subsubsection]{Corollary}
	\newtheorem*{coro*}{Corollary}

	\theoremstyle{definition}
	\newtheorem{defi}[subsubsection]{Definition}
	\newtheorem*{defi*}{Definition}
	
	\newtheorem{rema}[subsubsection]{Remark}
	\newtheorem{exam}[subsubsection]{Example}
	\newtheorem{nota}[subsubsection]{Notation}
	\newtheorem{cons}[subsubsection]{Construction}
	
	\numberwithin{equation}{subsubsection}
	
	\title[Maximal tori of monodromy groups of $F$-isocrystals and an application]{Maximal tori of monodromy groups of $F$-isocrystals and an application to abelian varieties}
	\date{\today}
	\makeatletter
	\@namedef{subjclassname@2020}{%
		\textup{2020} Mathematics Subject Classification}
	\makeatother
	


	\author{Emiliano Ambrosi and Marco D'Addezio}
	\address{Institut de Recherche Mathématique Avancée, Université de Strasbourg - 7 rue René Descartes, Office 207, 67084 Strasbourg (France)}
	\email{eambrosi@unistra.fr}
\address{Institut de Mathématiques de Jussieu-Paris Rive Gauche, Sorbonne Université - 4 place Jussieu, Case 247, 75005 Paris (France)}
\email{daddezio@imj-prg.fr}
	
	\begin{abstract}
		Let $X_0$ be a smooth geometrically connected variety defined over a finite field $\F_q$ and let $\mathcal E_0^{\dagger}$ be an irreducible overconvergent $F$-isocrystal on $X_0$. We show that if a subobject of minimal slope of the associated convergent $F$-isocrystal $\mathcal E_0$ admits a non-zero morphism to $\calO_{X_0}$ as a convergent isocrystal, then $\mathcal E_0^{\dagger}$ is isomorphic to $\calO^{\dagger}_{X_0}$ as an overconvergent isocrystal. This proves a special case of a conjecture of Kedlaya. The key ingredient in the proof is the study of the monodromy group of $\mathcal E_0^{\dagger}$ and the subgroup defined by $\mathcal E_0$. The new input in this setting is that the subgroup contains a maximal torus of the entire monodromy group. This is a consequence of the existence of a Frobenius torus of maximal dimension. As an application, we prove a finiteness result for the torsion points of abelian varieties, which extends the previous theorem of Lang--Néron and answers positively a question of Esnault. 
		
	\end{abstract}
	
	\maketitle

	\tableofcontents
	\section{Introduction}
	
	\subsection{Convergent and overconvergent isocrystals}Let $p$ be a prime number. The first Weil cohomology constructed to study varieties in characteristic $p$ was the \textit{$\ell$-adic \etale cohomology}, where $\ell$ is a prime different from $p$. Its associated category of coefficients is the category of \textit{lisse sheaves}. While $p$-adic \etale cohomology is not a Weil cohomology, moving from $\ell$ to $p$ one encounters two main $p$-adic cohomology theories: \textit{crystalline cohomology} and \textit{rigid cohomology}. These two give rise to different categories of “local systems”: \textit{convergent isocrystals} and\textit{ overconvergent isocrystals}. Let $\Fq$ be a finite field with $q$ elements, where $q$ is a power of $p$, and let $X_0$ be a smooth variety over $\Fq$. We write $\Fisoc(X_0)$ (resp. $\Foi(X_0)$) for the category of $\Qpbar$-linear convergent (resp. overconvergent) $F$-isocrystals over $X_0$. By \cite{Ked04}, these two categories are related by a natural fully faithful functor $\epsilon:\Foi(X_0)\rightarrow \Fisoc(X_0)$. When $X_0$ is proper, the functor $\epsilon$ is an equivalence. In general, the two categories have different behaviours. While $\Foi(X_0)$ shares many properties with the category of Weil lisse $\Qlbar$-sheaves, as explained in \cite{Dad} and \cite{Ked18}, the category $\Fisoc(X_0)$ has some exceptional $p$-adic features.
	
	For example, for every $\calE_0\in\Fisoc(X_0)$, after possibly shrinking $X_0$ to a dense open, there exists a filtration $$0=\mathcal E_0^{0}\subseteq \mathcal E_0^{1} \subseteq ...\subseteq \mathcal E_0^{n}=\mathcal E_0$$ where for each $i$ the quotient $\mathcal E_0^{i+1}/\mathcal E_0^{i}$ has uniquely slope $s_i$ at closed points and the sequence $s_1,\dots,s_n$ is increasing (see \cite{Katz79} and \cite[Corollary 4.2]{Ked16}). When $\calE_0=\epsilon(\calE_0^\dagger)$ for some $\calE_0^\dagger\in \Foi(X_0)$, the subobjects $\calE_0^i$ in general are not in the essential image of $\epsilon$ as well (see \cite[Remark 5.12]{Ked16}). Our main result highlights a new relationship between the subquotients of $\mathcal E^{\dagger}_0$ in $\Foi(X_0)$ and the ones of $\mathcal E_0^1$ in $\Fisoc(X_0)$.
	\begin{theo}[Theorem \ref{key:t}]\label{intro:t1}
		Let $\calE_0^{\dagger}$ be an irreducible $\Qpbar$-linear overconvergent $F$-isocrystal. If $\calE_0:=\epsilon(\calE_0^\dagger)$ admits a subobject of minimal slope $\calF_0\subseteq \calE_0$ with a non-zero morphism $\calF_0\to\calO_{X_0}$ of convergent isocrystals, then $\calE_0^{\dagger}$ has rank $1$.
	\end{theo}
 Theorem \ref{intro:t1} proves a particular case of the conjecture in \cite[Remark 5.14]{Ked16}. Even though  the conjecture turned out to be false in general, Theorem \ref{intro:t1} corresponds to the case when $\calF_1\subseteq \calE_1$ has minimal slope and $\calE_2$ is the convergent isocrystal $\calO_{X_0}$ endowed with some Frobenius structure (notation as in [\textit{ibid.}]).
 Recently Tsuzuki in \cite{Tsuzuki} proved (a modified version of) Kedlaya's conjecture over finite fields. In particular, he gave an alternative proof of Theorem \ref{intro:t1}. Our proof is different and independent. See Remark \ref{Relationships:r} for some more details on the differences.

	\subsection{Torsion points of abelian varieties}\label{perfect-p-tor:ss}
	Before explaining the main ingredients of the proof of Theorem \ref{intro:t1}, let us describe an application to torsion points of abelian varieties. This was our main motivation to prove Theorem \ref{intro:t1}.
	Let $\F$ be an algebraic closure of $\Fq$ and $\F\subseteq k$ be a finitely generated field extension. For an abelian variety $A$ over $k$ we write $\Tr_{k/\F}(A)$ for its \textit{$k/\F$-trace} (cf. \cite[§6]{Conrad}). Recall the classical Lang--Néron theorem (see \cite{LN} or \cite{Conrad}).
	\begin{theo}[Lang--Néron] \label{Lang-Neron:t}
		If $A$ is an abelian variety over $k$ such that $\Tr_{k/\F}(A)=0$, then the group $A(k)$ is finitely generated.  
	\end{theo}
	By Theorem \ref{Lang-Neron:t}, if we denote by $A^{(n)}$ the Frobenius twist of $A$ by the $p^n$-power Frobenius, we have a tower of finite groups $A(k)_{\textrm{tors}}\subseteq A^{(1)}(k)_{\textrm{tors}}\subseteq A^{(2)}(k)_{\textrm{tors}}\subseteq\dots .$
	In June 2011, in a correspondence with Langer and R\"ossler, Esnault asked whether this chain is \textit{eventually stationary}. Since $$\bigcup_{n\geq 0} A^{(n)}(k)=A(k^{\perf}),$$ where $k^{\perf}$ is a perfect closure of $k$, an equivalent way to formulate the question is to ask whether the group of $k^\perf$-rational torsion points $A(k^{\perf})_{\tors}$ is a finite group. As an application of Theorem \ref{intro:t1}, we give a positive answer to her question.
	\begin{theo}[Theorem \ref{perfect-p-tor:t}]\label{intro-perfect-p-tor:t}
		If $A$ is an abelian variety over $k$ such that $\Tr_{k/\F}(A)=0$, then the group $A(k^{\mathrm{perf}})_{\mathrm{tors}}$ is finite.
	\end{theo}
	
	\begin{rema}\label{intro-p-tor:r}
		Theorem \ref{intro-perfect-p-tor:t} was already known for elliptic curves, by the work of Levin in \cite{Lev68}, and for ordinary abelian varieties, by \cite[Theorem 1.4]{Ros17}. When $\ell$ is a prime different from $p$, the group $A[\ell^{\infty}]$ is étale, hence $A[\ell^{\infty}](k^{\mathrm{perf}})=A[\ell^{\infty}](k)$. Therefore, in Theorem \ref{intro-perfect-p-tor:t}, the finiteness of torsion points of prime-to-$p$ order is guaranteed by Theorem \ref{Lang-Neron:t}.
	\end{rema}
	In order to deduce Theorem \ref{intro-perfect-p-tor:t} from Theorem \ref{intro:t1} we use the \textit{crystalline Dieudonné theory}, as developed in \cite{BBM82}. The proof of Theorem \ref{intro-perfect-p-tor:t} is by contradiction. If $|A[p^{\infty}](k^{\mathrm{perf}})|=\infty$, then there exists a monomorphism $\Qp/\Zp\hookrightarrow A[p^\infty]^{\textrm{\'et}}$ from the trivial $p$-divisible group $\Qp/\Zp$ over $k$ to the étale part of the $p$-divisible group of $A$. Spreading out to a “nice” model $\mathfrak A/X$ of $A/k$ and applying the contravariant crystalline Dieudonné functor $\mathbb D$, one gets an epimorphism of $F$-isocrystals $\mathbb D(\mathfrak A[p^\infty]^{\textrm{\'et}})\twoheadrightarrow \mathbb D((\Qp/\Zp)_X)\simeq \mathcal O_X$ over $X$. By a descent argument and Theorem \ref{intro:t1}, the quotient extends to a quotient $\mathbb D(\mathfrak A[p^\infty])\twoheadrightarrow\mathcal O_X$ over $X$. Going back to $p$-divisible groups, this gives an injective map $\Qp/\Zp\hookrightarrow A[p^\infty]$ over $k$. Therefore, $A[p^\infty](k)$ would be an infinite group, contradicting Theorem \ref{Lang-Neron:t}.
	\subsection{Monodromy groups}If $X_0$ is geometrically connected over $\Fq$, the categories $\Fisoc(X_0)$ and $\Foi(X_0)$ and their versions without Frobenius structures $\Isoc(X_0)$ and $\oi(X_0)$ are neutral Tannakian categories. The choice of an $\F$-point $x$ of $X_0$ induces fibre functors for all these categories. To prove Theorem \ref{intro:t1}, we study the \textit{monodromy groups} associated to the objects involved. For every $\calE_0^\dagger\in \Foi(X_0)$, we have already seen that we can associate an object $\calE_0:=\epsilon(\calE_0^\dagger)\in\Fisoc(X_0)$. We denote by $\mathcal E^{\dagger}\in \Isoc^{\dagger}(X_0)$ (resp. $\mathcal E\in \Isoc(X_0)$\footnote{We point out that the convention on the subscript $_0$ is not consistent between $F$-isocrystals and varieties, namely $\mathcal E$ denotes an isocrystal over $X_0$ and not over $X$.}) the isocrystal obtained from $\mathcal E_0^{\dagger}$ (resp. $\mathcal E_0$) by forgetting its Frobenius structure. Using the Tannakian formalism, we associate to each of these objects an algebraic group $G(-)$. They all sit naturally in a commutative diagram of closed immersions
	\begin{center}
		\begin{tikzcd}
			G(\calE)\arrow[hook]{r}\arrow[hook,d] & G(\calE_0)\arrow[hook,d]\\
			G(\calE^{\dagger})\arrow[hook]{r} & G(\calE_0^{\dagger}).
		\end{tikzcd}
	\end{center}
	While $G(\calE_0)$ and $G(\calE_0^{\dagger})$ are the analogues of the arithmetic monodromy group of a lisse sheaf, $G(\calE)$ and $G(\calE^{\dagger})$ correspond to the geometric monodromy group.
	If $\mathcal E_0^{\dagger}$ is irreducible and its determinant has finite order, as a consequence of class field theory, the group $G(\mathcal E_0^{\dagger})/G(\mathcal E^{\dagger})$ is finite, \cite[Theorem 3.4.7]{Dad}. We prove that the same is true for $G(\mathcal E_0)/G(\mathcal E)$.

	\begin{prop}[Proposition \ref{constant-objects:p}]
		\label{intro-constant-objects:p}
		Let $\calE^{\dagger}_0$ be an irreducible overconvergent $F$-isocrystal with finite order determinant. The quotient $G(\calE_0)/G(\calE)$ is finite.
	\end{prop}
	To prove Proposition \ref{intro-constant-objects:p}, we have to show that $G(\calE)$ is “big”. We study $G(\calE)$ as a subgroup of $G(\calE^\dagger)$ and we prove our fundamental result.
	\begin{theo}[Theorem \ref{rank:t}]\label{intro:t2}
		If $\calE^{\dagger}_0$ is an overconvergent $F$-isocrystal, then $G(\calE)$ contains a maximal torus of $G(\mathcal E^{\dagger})$.
	\end{theo}
	To prove Theorem \ref{intro:t2}, we use the existence of \textit{Frobenius tori} which are maximal tori of $G(\mathcal E_0^{\dagger})$ (Theorem \ref{frob-tori:t}). First we reduce to the case when $\calE^{\dagger}_0$ is semi-simple and \textit{algebraic} (cf. Definition \ref{alg:d}). By Theorem \ref{frob-tori:t}, there exists a closed point $i_0:x_0\hookrightarrow X_0$ such that the subgroup $G(i_0^*\mathcal E^{\dagger}_0)\subseteq G(\mathcal E^{\dagger}_0)$ contains a maximal torus of $G(\mathcal E^{\dagger}_0)$. Since over a closed point every $F$-isocrystal admits an overconvergent extension, one has $G(i_0^*\mathcal E^{\dagger}_0)=G(i_0^*\mathcal E_0)$. Hence, $G(\mathcal E_0)$ contains a maximal torus of $G(\mathcal E^{\dagger}_0)$. To pass from $G(\mathcal E_0)\subseteq G(\mathcal E^{\dagger}_0)$ to $G(\mathcal E)\subseteq G(\mathcal E^{\dagger})$, we will apply Theorem \ref{frob-tori:t} to an auxiliary overconvergent $F$-isocrystal $\widetilde{\mathcal E}^{\dagger}_0$ over $X_0$, such that $G(\widetilde{\mathcal E}^{\dagger})=G(\mathcal E^{\dagger})$, $G(\widetilde{\mathcal E})=G(\mathcal E)$ and with the additional property that the natural map $G(\widetilde{\mathcal E}_0)/G(\widetilde{\mathcal E})\rightarrow G(\widetilde{\mathcal E}_0^{\dagger})/G(\widetilde{\mathcal E}^{\dagger})$ is an isomorphism. 
	
	\begin{rema}
		In \cite[page 460]{CrewMon} Crew asks whether, under the assumptions of Theorem \ref{intro:t2}, the group $G(\mathcal E)$ is a parabolic subgroup of $G(\mathcal E^{\dagger})$. In two subsequent articles \cite{Crewab} and \cite{Crewsh}, he gives a positive answer to his question in some particular cases. Since parabolic subgroups of reductive groups always contain a maximal torus, Theorem \ref{intro:t2} is evidence for Crew's expectation.
	\end{rema}
	To deduce Theorem \ref{intro:t1} from Proposition \ref{intro-constant-objects:p}, we first reduce ourself to the situation where the determinant of $\mathcal E_0^{\dagger}$ has finite order. To simplify, let us assume that $\mathcal E_0$ has constant Newton polygon, that $\mathcal F_0$ is equal to $\mathcal E^{1}_0$, the subobject of minimal slope, and that $G(\mathcal E_0)$ is connected. Proposition \ref{intro-constant-objects:p} implies that $G(\mathcal E)=G(\mathcal E_0)$ hence that the morphism $\calE^1_0\to\calO_{X_0}$ commutes with the trivial Frobenius structure on $\calO_{X_0}$. In particular, $\calE^1_0$ has slope $0$, so that the minimal slope of $\calE_0$ is $0$. Since the determinant of $\mathcal E_0$ has finite order, this implies that $\mathcal E_0^1=\mathcal E_0$ hence that $\mathcal E_0$ admits a quotient $\mathcal E_0\twoheadrightarrow\calO_{X_0}$ in $\Fisoc(X_0)$. As $\epsilon:\Foi(X_0)\rightarrow \Fisoc(X_0)$ is fully faithful, $\calE_0^\dagger$ admits a quotient $\mathcal E_0^\dagger\twoheadrightarrow\calO_{X_0}^\dagger$ in $\Foi(X_0)$. On the other hand, $\mathcal E^{\dagger}_0$ is irreducible, so that the quotient gives actually an isomorphism $\mathcal E^{\dagger}_0\simeq\calO_{X_0}^\dagger$.
	
	\subsection{Weak (weak) semi-simplicity}\label{weak:ss} As an additional outcome of Theorem \ref{intro:t2}, we get a semi-simplicity result for extensions of \textit{constant} convergent $F$-isocrystals (cf. Definition \ref{cst:d}). Recall that an overconvergent $F$-isocrystal $\calE_0^\dagger$ is said \textit{pure of weight $n$}, where $n$ is an integer, if for every closed point of $X_0$ of degree $d$ over $\Fq$ the eigenvalues of the Frobenii at closed points (cf. \cite[Definition 9.5]{Ked16}) have complex absolute value $q^{dn/2}$ for every isomorphism $\overline \Q_p\simeq \C$. Let $\Fisoc_{\pure^{\dagger}}(X_0)$ denote the Tannakian subcategory of $\Fisoc(X_0)$ generated by the essential image via $\epsilon:\Foi(X_0)\rightarrow \Fisoc(X_0)$ of pure objects in $\Foi(X_0)$. Thanks to a group-theoretic argument (Lemma \ref{coh-van:l}), Theorem \ref{intro:t2} implies the following.
	\begin{coro}[Corollary \ref{weakweak:c}]\label{intro-weakweak:c}
		A convergent $F$-isocrystal in $\Fisoc_{\pure^{\dagger}}(X_0)$ which is an extension of constant $F$-isocrystals is constant.
	\end{coro}

For every smooth and proper morphism $f_0:Y_0\to X_0$ and every $i\in \N$, the subquotients of the higher direct image $R^if_{0,\crys*}\calO_{Y_0}$ are in $\Fisoc_{\pured}(X_0)$ by \cite{KM74} and \cite{Shiho} (see \cite[Fact 3.1.1.2 and Fact 3.2.1.1]{mioneron}). Therefore, Corollary \ref{intro-weakweak:c} applies to these convergent $F$-isocrystals. In this text we will say that these convergent $F$-isocrystals \textit{come from geometry}.

		Using Artin--Schreier--Witt theory, one can construct on $\mathbb{A}^1_{\Fq}$ non-constant extensions of constant unit-root convergent $F$-isocrystals. One can further construct these extensions in such a way that the resulting convergent $F$-isocrystal has \textit{log-decay}, in the sense of \cite{Joe}. Corollary \ref{intro-weakweak:c} shows, for example, that these $F$-isocrystals are outside $\Fisoc_{\pured}(\mathbb{A}^1_{\Fq})$.

	\begin{rema}
		Let $\calE_0$ be a convergent $F$-isocrystal with constant Newton polygons. Corollary \ref{intro-weakweak:c} implies that $G(\mathcal E)$ has no unipotent quotients. Let $\calE^1$ be the convergent isocrystal which underlies the subobject of $\calE_0$ of minimal slope. Since $G(\mathcal E^1)$ is a quotient of $G(\mathcal E)$, it does not have unipotent quotients as well. In \cite[Conjecture 7.4 and Remark 7.4.1]{Chai}, Chai conjectured that if $\mathcal E_0^{\dagger}$ is the higher direct image of a family of ordinary abelian varieties, then $G(\mathcal E^{1})$ is reductive. Corollary \ref{intro-weakweak:c} may be thought as a first step towards his conjecture.
	\end{rema}
	\subsection{Organization of the paper}
	In §\ref{monodromy:s} we recall the definition of the monodromy groups of the various categories of isocrystals and we prove Theorem \ref{intro:t2}. In §\ref{KC:s} we prove Theorem \ref{intro:t1} and some of its consequences. Finally, in §\ref{perfect-p-tor:s} we prove Theorem \ref{intro-perfect-p-tor:t}.
	\subsection{Acknowledgements} We learned about the problem of perfect torsion points on abelian varieties reading a question of Damian Rössler on the website Mathoverflow \cite{Ros11}. We would like to thank him and Hélène Esnault for their interest and comments on our result. We are grateful to Brian Conrad and Michel Brion for some enlightening discussions about epimorphic subgroups and maximal rank subgroups of reductive groups. We also thank Brian Conrad for the references \cite{Borel1} and \cite{Brion}. We thank Simon Pepin Lehalleur for pointing out a simpler proof of Lemma \ref{coh-van:l} and Raju Krishnamoorthy for some discussions on the crystalline Dieudonné module functor. We thank Anna Cadoret and Hélène Esnault for suggesting several expository improvements. Finally, we thank the anonymous referees for many thoughtful comments.
	\subsection{Notation}\label{notation:s}
	\subsubsection{}\label{tannakian:ss} Let $\KK$ be a characteristic $0$ field and $\bfC$ a $\KK$-linear Tannakian category. A \textit{Tannakian subcategory} of $\bfC$ is a strictly full subcategory of $\bfC$ closed under direct sums, tensor products, duals and subobjects. For $\calE\in\bfC$, we denote by $\langle \calE \rangle$ the smallest Tannakian subcategory of $\bfC$ containing $\calE$. Let $\omega:\bfC\rightarrow \KK$ be a fibre functor. For every $\calE\in\bfC$, the restriction of $\omega$ to $\langle \calE \rangle$ defines a fibre functor of $\langle \calE \rangle$. We denote by $G(\calE)$ the Tannakian group of $\langle \calE \rangle$ with respect to this fibre functor. In general, the fibre functor will be clear from the context, so that we do not keep $\omega$ in the notation. The group $G(\calE)$ is called the \textit{monodromy group} of $\calE$. If $G(\calE)$ is finite we say that $\calE$ is \textit{finite}.
	\subsubsection{}
	Let $G$ be an algebraic group over $\KK$. We denote by $G^{\circ}$ the connected component of the identity of $G$ and by $\rk(G)$ the \textit{reductive rank} of $G$, namely the dimension of the maximal tori of $G$. We say that a subgroup $H$ of $G$ is of \textit{maximal rank} if $\rk(H)=\rk(G)$. Besides, we write $X^*(G)$ for the group of characters of $G$. When $\KK$ is a characteristic $0$ field and $f:G\to H$ is a morphism of affine group schemes over $\KK$,  we say that $f$ is \textit{injective} if it is a closed immersion and we say that $f$ is \textit{surjective} if it is faithfully flat. Since over a characteristic $0$ field every affine group scheme is reduced, this should not generate any confusions.
	\section{Monodromy of convergent isocrystals}\label{monodromy:s}
	
	\subsection{Review of isocrystals}
	\label{review-on-iso:ss}We recall in this section some basic facts about isocrystals. See \cite[§2]{Ked16} for more details.
	Throughout §\ref{review-on-iso:ss}, let $\kappa$ be a subfield of $\F$. We denote by $W(\kappa)$ the ring of Witt vectors of $\kappa$ and by $K(\kappa)$ its field of fractions. We write $\Qpbar$ for a fixed algebraic closure of $\Qp$, and we suppose chosen an embedding of $W(\F)$ in $\Qpbar$. Let $Y$ be a smooth variety over $\kappa$.
	\begin{defi}\label{isoc:d}We write $\Isoc(Y/K(\kappa))$ for the category of \textit{convergent isocrystals} over $Y$ with respect to $K(\kappa)$. For every finite field extension $K(\kappa)\subseteq L$, we have a category of convergent isocrystals over $Y$ endowed with an $L$-structure, \cite[§1.4.1]{Abe}, denoted by $\Isoc(Y/K(\kappa))_L$. The category $\Isoc(Y)$ of \textit{$\Qpbar$-linear convergent isocrystals} is defined to be the 2-inductive limit of the categories $\Isoc(Y/K(\kappa))_L$ where $L$ varies among the finite extensions of $K(\kappa)$ in $\Qpbar$. We write $\calO_{Y}$ for the convergent isocrystal associated to the crystalline structure sheaf. We will also consider the category of \textit{$\Qpbar$-linear convergent $F$-isocrystals}, denoted by $\Fisoc(Y)$. This category consists of pairs $(\calE,\Phi)$, where $\calE$ is a $\Qpbar$-linear convergent isocrystal and $\Phi$ is a \textit{Frobenius structure} on $\calE$, namely a $\Qpbar$-linear isomorphism $F^*\calE\iso \calE$ where $F$ is the absolute Frobenius.
	\end{defi}
	The category of $\Qpbar$-linear convergent $F$-isocrystals has also a different incarnation.
	\begin{defi}
	Let $\Crys(Y/W(\kappa))$ be the category of crystals of finite $\calO_{Y,\crys}$-modules and $\Crys(Y/K(\kappa))$ its isogeny category. As above, one can extend the field of scalars of $\Crys(Y/K(\kappa))$ to $\Qpbar$ obtaining the category $\Crys(Y/K(\kappa))_{\Qpbar}$. We write $\Fcrys(Y/K(\kappa))_{\Qpbar}$ for the category of objects in $\Crys(Y/K(\kappa))_{\Qpbar}$ endowed with a Frobenius structure.
	\end{defi}

	\begin{theo}[Ogus, Berthelot]\label{Ogus:t} There exists a canonical equivalence of categories $$\Fcrys(Y/K(\kappa))_{\Qpbar}\iso \Fisoc(Y).$$
	\end{theo}
	\begin{proof}
	By \cite[Th\'{e}or\`{e}me 2.4.2]{Ber96}, there exists an equivalence of categories
		$$\Fcrys(Y/K(\kappa))\simeq \Fisoc(Y/K(\kappa)).$$
		This implies that for every finite field extension $K(\kappa)\subseteq L$ there also exists an equivalence $$\Fcrys(Y/K(\kappa))_{L}\simeq \Fisoc(Y/K(\kappa))_{L}.$$
	Taking the 2-inductive limit over all the finite field extensions $K(\kappa)\subseteq L$ we conclude the proof.
	\end{proof}
	\begin{rema}
		In light of Theorem \ref{Ogus:t}, we will feel free to refer to convergent $F$-isocrystals simply as $F$-isocrystals. Besides, in what follows we will mostly work with $\Qpbar$ coefficients. Therefore, when we will talk about ($F$-)isocrystals we will actually mean $\Qpbar$-linear convergent ($F$-)isocrystals except when it is explicitly said differently.
	\end{rema}
	\begin{defi}\label{oi:d}
		Let $\oi(Y)$ be the category of \textit{($\Qpbar$-linear) overconvergent isocrystals} and $\Foi(Y)$ the category of \textit{($\Qpbar$-linear) overconvergent $F$-isocrystals}. The overconvergent isocrystal associated to the structural sheaf will be denoted by $\calO_Y^\dagger$.
	\end{defi}
The categories of convergent and overconvergent $F$-isocrystals are related by a natural functor $\epsilon: \Foi(Y)\rightarrow \Fisoc(Y)$.
	\begin{theo}[\cite{Ked04}]\label{kedlaya:t}
		The functor $\epsilon:\Foi(Y)\rightarrow \Fisoc(Y)$ is fully faithful.
	\end{theo}

	\begin{rema}\label{epimorphic:r}
		Even though $\epsilon$ is fully faithful, the essential image is not closed under subquotients. Therefore, the essential image is not a Tannakian subcategory in the sense of §\ref{tannakian:ss}, so that the induced morphism on Tannakian groups is not surjective. Nevertheless, the morphism is an \textit{epimorphism} in the category of affine group schemes (cf. \cite{Borel1}). See also Remark \ref{sl3:r} for further comments.
	\end{rema}

	\begin{defi}
		Suppose that $Y$ is connected and let $\calE$ be an $F$-isocrystal of rank $r$. We denote by $\{a_i^{\eta}(\calE)\}_{1\leq i \leq r}$ the set of generic slopes of $\calE$. We use the convention that $a_1^{\eta}(\calE)\leq \dots \leq a_r^{\eta}(\calE)$, thus the choice of the ordering does not agree with \cite{DK16}. We say that $\calE$ is \textit{isoclinic} if $a_1^{\eta}(\calE)=a_r^{\eta}(\calE)$. A subobject $\calF$ of $\calE$ is \textit{of minimal slope} if it is isoclinic of slope $a_1^{\eta}(\calE)$. See \cite[§3 and §4]{Ked16} for more details about slopes.
	\end{defi}

	\subsection{The fundamental exact sequence}
	\label{fund-exact-sequence:ss}
	We shall briefly review the theory of \textit{monodromy groups} of $F$-isocrystals. These monodromy groups have been studied at the beginning by Crew in \cite{CrewMon}. More recent work can be found in \cite{Dad}. In Proposition \ref{conv-oconv-mon:p}, we recall a fundamental diagram of monodromy groups that we will extensively use in the next sections. 
	
	\begin{nota}\label{fin-fields:nota}
		Let $X_0$ be a smooth geometrically connected variety over $\Fq$. We choose once and for all an $\F$-point $\tilde{x}$ of $X_0$. This defines fibre functors for all the Tannakian categories of isocrystals previously defined. We write $\mathbbm 1_0^{\dagger}$ for the overconvergent $F$-isocrystal $\calO_{X_0}^\dagger$ endowed with its canonical Frobenius structure. For every $\calE_0^{\dagger}\in \Foi(X_0)$ we consider three associated objects. We denote by $\calE^\dagger\in \oi(X_0)$ the overconvergent isocrystal obtained from $\calE_0^{\dagger}$ by forgetting the Frobenius structure. The image of $\calE_0^{\dagger}$ in $\Fisoc(X_0)$ will be denoted by removing the superscript $^\dagger$. At the same time, $\calE$ will be the convergent isocrystal in $\Isoc(X_0)$, obtained from $\calE_0\in \Fisoc(X_0)$ by forgetting its Frobenius structure. Here a summary table.

		\renewcommand{\arraystretch}{1.5}
		\begin{center}
			\vskip1.8em
			
			\begin{tabular}{|lc||c|c|}
				
				\hline
				&&Isocrystal &$F$-isocrystal\\
				\hline\hline
				&Convergent&
				$\calE$&$\calE_0$
				
				\\ \hline
				&Overconvergent\ \ \ &$\calE^\dagger$&$\calE_0^\dagger$
				\\ \hline
			\end{tabular}
			\vskip1.8em
		\end{center}
For each of these objects we have a monodromy group $G(-)$ (see §\ref{tannakian:ss}) with respect to the fibre functor associated to our $\F$ point $\tilde{x}$.
	\end{nota}
	\begin{defi}\label{cst:d}
		
		We say that a convergent isocrystal is \textit{trivial} if it is isomorphic to $\mathbbm 1^{\oplus r}$ for some $r\in \N$. An $F$-isocrystal $\calE_0$ is said \textit{constant} if the convergent isocrystal $\calE$ is trivial. We denote by $\Fisoc_{\cst}(X_0)$ the strictly full subcategory of $\Fisoc(X_0)$ of constant objects. For $\calE_0\in \Fisoc(X_0)$, we denote by $\langle \calE_{0} \rangle_{\cst}\subseteq \langle \calE_{0} \rangle$ the Tannakian subcategory of constant objects and by $G(\calE_0)^\cst$ the Tannakian group of $\langle \calE_{0} \rangle_{\cst}$. Finally, for $\alpha\in \Qpbar$ and $\calE_0\in \Fisoc(X_0)$, we denote by $\calE_0^{(\alpha)}$ the $F$-isocrystal obtained from $\calE_0$ multiplying its Frobenius structure by $\alpha$. We will call $\calE_0^{(\alpha)}$ the \textit{twist} of $\calE_0$ by $\alpha$. We give analogous definitions for overconvergent isocrystals.
	\end{defi}

	\begin{rema}
		\label{cst:r}
		The category $\Fisoc_{\cst}(X_0)$ is a Tannakian subcategory of $\Fisoc(X_0)$ in the sense of §\ref{tannakian:ss}. Let $p_{X_0}:X_0\to \Spec(\Fq)$ be the structural morphism of $X_0$. Every constant $F$-isocrystal is the inverse image via $p_{X_0}$ of an $F$-isocrystal defined over $\Spec(\Fq)$. Since $X_0$ is geometrically connected over $\Fq$, the functor $p_{X_0}^*$ is fully faithful, thus the same is true for $\Fisoc_{\cst}(X_0)$. The category $\Fisoc(\Spec(\Fq))$ is equivalent to the category of $\Qpbar$-vector spaces endowed with a linear automorphism (induced by the Frobenius structure). Hence $\Fisoc(\Spec(\Fq))$ is equivalent to the category of $\Qpbar$-linear representation of $\Z$. Finally, we recall that since $\Foi(\F_q)=\Fisoc(\F_q)$, the natural functor $\epsilon:\Foi(X_0)\rightarrow \Fisoc(X_0)$ induces an equivalence of categories between $\Foi_{\cst}(X_0)$ and $\Fisoc_{\cst}(X_0)$.
	\end{rema}
	The following proposition shows that these monodromy groups fit into exact sequences, analogous to the fundamental exact sequence relating the arithmetic and the geometric monodromy groups of a lisse sheaf. This holds in general for \textit{neutral Tannakian categories with Frobenius}, \cite[Appendix A]{Dad}. See also \cite[Corollary 1.6]{DE20}.
	\begin{prop}\label{conv-oconv-mon:p}\
			Let $X_0$ be a smooth geometrically connected variety over $\Fq$ and let $\calE_0^\dagger$ be an overconvergent $F$-isocrystal over $X_0$. There exists a functorial commutative diagram 
		\begin{equation}
		\begin{tikzcd}
		1\arrow{r} & G(\calE)\arrow{r}\arrow[hook,d] & G(\calE_0)\arrow{r}\arrow[hook,d] & G(\calE_0)^{\cst}\arrow{r}\arrow[d,two heads] &1\\
		1\arrow{r} & G(\calE^{\dagger})\arrow{r} & G(\calE_0^{\dagger})\arrow{r} & G(\calE_0^{\dagger})^{\cst}\arrow{r} & 1
		\end{tikzcd}
		\end{equation}
		with exact rows. The left and the central vertical arrows are injective and the right one is surjective. Moreover, $G(\calE_0)^{\cst}$ and $G(\calE_0^{\dagger})^{\cst}$ are commutative algebraic groups.
	\end{prop}
	\begin{proof}
		The inverse image functor with respect to the $q$-power Frobenius of $X_0$, is an equivalence of categories both for the convergent and overconvergent isocrystals over $X_0$ (see \cite[Corollary 4.10]{Ogu84} and \cite{Laz17}). The exactness of the rows then follows from \cite[A.2.2.(iii)]{Dad}. In addition, the right vertical arrow is surjective because, by the discussion in Remark \ref{cst:r}, the functor $\langle \calE_0^\dagger \rangle_{\cst}\to\langle \calE_0 \rangle_{\cst}$ is fully faithful and the essential image is closed under subquotients. Finally, by \cite[A.2.2.(iv)]{Dad}, the algebraic groups $G(\calE_0)^{\cst}$ and $G(\calE_0^{\dagger})^{\cst}$ are commutative.
	\end{proof}
	\begin{rema}\label{dont-factor:r}
		We do not know whether the natural quotient $ \varphi :G(\calE_0)^{\cst}\twoheadrightarrow G(\calE^{\dagger}_0)^{\cst}$ is an isomorphism in general. Via the Tannakian formalism, to prove the injectivity of $ \varphi $, one has to show that the embedding $\langle \calE_{0}^\dagger \rangle_{\cst}\hookrightarrow \langle \calE_{0} \rangle_{\cst}$ is essentially surjective. While every $\mathcal F_0\in\langle \calE_{0} \rangle_{\cst}$ comes from an object $\calF_0^\dagger$ in $\Foi(X_0)$, we do not know whether such an $\mathcal F_0^\dagger$ lies in $\langle \calE_{0}^\dagger \rangle$. A priori, it might happen that a non-constant subobject $\mathcal F_0\subseteq \mathcal E_0$ admits a constant quotient $\mathcal F_0 \twoheadrightarrow \calT_0$. In this case it is unclear whether $\mathcal T_0$ is in the essential image of  $\langle \calE_{0}^\dagger \rangle_{\cst}\hookrightarrow \langle \calE_{0} \rangle_{\cst}$. This will be the main issue in the proof of Theorem \ref{rank:t}. We bypass the problem by embedding $\calE_0^\dagger$ in an auxiliary overconvergent $F$-isocrystal $\widetilde{\calE}_0^\dagger$ with $G(\widetilde{\calE}_0)^{\cst}\simeq G(\widetilde{\calE}^{\dagger}_0)^{\cst}$. As an application of Theorem \ref{rank:t}, we will also prove in Corollary \ref{red-rank-cst:c} that if $\mathcal E^{\dagger}_0$ is algebraic (cf. §\ref{alg:d}) and semi-simple, then $ \varphi:G(\calE_0)^{\cst}\twoheadrightarrow G(\calE^{\dagger}_0)^{\cst}$ is an isogeny of linear algebraic groups.
	\end{rema}

	\subsection{Maximal tori}
	\label{max-tori:ss}
	In this section, we briefly recall the main theorem on \textit{Frobenius tori} of overconvergent $F$-isocrystals in \cite[§4.2]{Dad} and we use it to prove Theorem \ref{rank:t}. For this task, the main issue is to pass from the arithmetic situation (Corollary \ref{arit-rank:c}) to the geometric one (Theorem \ref{rank:t}). We keep the notation as in §\ref{fin-fields:nota}

\begin{defi}\label{to:d}Let $i_0:x_0\hookrightarrow X_0$ be the closed immersion of a closed point of $X_0$. For every overconvergent $F$-isocrystal $\calE_0^\dagger$ we have an inclusion $G(i_0^*\calE_0^\dagger) \hookrightarrow G(\calE_0^\dagger)$, with $G(i_{0}^*\calE_0^\dagger)$ commutative. The image of the maximal torus of $G(i_{0}^*\calE_0^\dagger)$ in $G(\calE_0^\dagger)$ is the \textit{Frobenius torus} of $\calE_{0}^\dagger$ at $x_0$, denoted by $T_{x_0}(\calE_0^\dagger)$.
\end{defi}


Thanks to Deligne's conjecture for lisse sheaves and overconvergent $F$-isocrystals, for a certain class of overconvergent $F$-isocrystals is possible to construct \textit{$\ell$-adic companions} where $\ell$ is a prime different from $p$, \cite{Abeesnault}. From this construction one can translate results known for lisse sheaves to overconvergent $F$-isocrystals. Theorem \ref{frob-tori:t} is an example of such a technique (see also §\ref{Langlands:r}). For the existence of companions one needs some mild assumptions on the eigenvalues of the Frobenii at closed points.

\begin{defi}\label{alg:d}		
	An overconvergent $F$-isocrystal $\calE_0^\dagger$ is \textit{algebraic} if the eigenvalues of the Frobenii at closed points (cf. \cite[Definition 9.5]{Ked16}) are algebraic numbers.
	\end{defi}

	\begin{theo}[{\normalfont\cite[Theorem 4.2.11]{Dad}}]\label{frob-tori:t} Let $\calE_0^\dagger$ be an algebraic overconvergent $F$-isocrystal over $X_0$. There exists a Zariski-dense set of closed points $x_0$ of $X_0$ such that the torus $T_{x_0}(\calE_0^\dagger)$ is a maximal torus of $G(\calE_0^\dagger)$.
	\end{theo}

\begin{rema}
	It is worth mentioning that when $\mathcal E_0^\dagger$ is pure Theorem \ref{frob-tori:t} is also a consequence of the new crystalline Čebotarev density theorem proven by Hartl and P\'al in \cite[Theorem 12.2]{HP}. Note that $\mathcal E_0^\dagger$ is pure if it comes from geometry (cf. §\ref{weak:ss}) or, by \cite[Theorem 2.7]{Abeesnault}, if it is irreducible with finite order determinant. 
\end{rema}

	\begin{coro}\label{arit-rank:c}
		Let $\calE_0^\dagger$ be an algebraic overconvergent $F$-isocrystal. The closed subgroup $G(\calE_0)\subseteq G(\calE^{\dagger}_0)$ is a subgroup of maximal rank.
	\end{coro}
	\begin{proof}
	 Thanks to Theorem \ref{frob-tori:t}, we can find a closed embedding of a closed point $i_0:x_0\hookrightarrow X_0$ such that $T_{x_0}(\calE_0^\dagger)$ is a maximal torus of $G(\calE_0^\dagger)$.
	 We have a commutative diagram
		\begin{center}
			\begin{tikzcd}
				G(i_{0}^*\calE_0)\arrow[hook,r]\arrow["\sim" labl,d] & G(\calE_0)\arrow[hook,d]\\
				G(i_{0}^*\calE^\dagger_0)\arrow[hook,r] & G(\calE^{\dagger}_0),\\
			\end{tikzcd}
		\end{center}
where the morphism $G(i_0^*\calE_0)\to G(i_0^*\calE_{0}^\dagger)$ is an isomorphism by Remark \ref{cst:r}. Since $G(i_{0}^*\calE^\dagger_0)$ is a subgroup of $G(\calE_0^\dagger)$ of maximal rank, the same is true for the subgroup $G(\calE_0)\subseteq G(\calE_0^\dagger)$. 
	\end{proof}

	\begin{coro}\label{mult-type:c}
		If $\calE_{0}^\dagger$ is an algebraic semi-simple overconvergent $F$-isocrystal, then $G(\calE_0)^{\cst}$ and $G(\calE^{\dagger}_0)^{\cst}$ are groups of multiplicative type.
	\end{coro}
	\begin{proof}
		By Proposition \ref{conv-oconv-mon:p}, the algebraic groups $G(\calE^{\dagger}_0)^{\cst}$ and $G(\calE_0)^{\cst}$ are commutative. It suffices to verify that they are also reductive. The former is a quotient of $G(\calE_{0}^\dagger)$, which is reductive because $\calE_{0}^\dagger$ is semi-simple. The latter is a quotient of $G(\calE_0)$, which by Corollary \ref{arit-rank:c} is a subgroup of $G(\calE_{0}^\dagger)$ of maximal rank. Let $R_u(G(\calE_0)^{\cst})$ be the unipotent radical of $G(\calE_0)^{\cst}$. Since $G(\calE_0)^{\cst}$ is commutative, $R_u(G(\calE_0)^{\cst})$ is a quotient of $G(\calE_0)$. Thus $R_u(G(\calE_0)^{\cst})$ is trivial by the group-theoretic Lemma \ref{coh-van:l} below. This concludes the proof. 
	\end{proof}

	\begin{rema}
	Note that even the $\ell$-adic analogue of Corollary \ref{mult-type:c} is true and it follows easily from the homotopy exact sequence of the étale fundamental group.
\end{rema}
	\begin{lemm}
		\label{coh-van:l}
		Let $\KK$ be an algebraically closed field of characteristic $0$, let $G$ be a reductive group over $\KK$ and let $H$ be a subgroup of $G$ of maximal rank. Every morphism from $H$ to a unipotent group is trivial. Equivalently, the group $\Ext^1_H(\mathbb K,\mathbb{K})$ vanishes.
	\end{lemm}
	\begin{proof}
		Every unipotent group is an iterated extension of copies of $\mathbb G_a$. Therefore, it is enough to show that every morphism from $H$ to $\mathbb G_a$ is trivial. Suppose there exists a non-trivial morphism $\varphi: H\to \mathbb{G}_a$. As $\mathrm{char}(\KK)=0$, the image of $\varphi$ is $\Ga$ itself. We write $K$ for the kernel of $\varphi$ and $N_{G}(K^{\circ})$ (resp. $N_{G}(K)$) for be the normaliser of $K^{\circ}$ (resp. $K$) in $G$. Every map from a torus to $\Ga$ is trivial, thus the subgroup $K\subseteq G$ has maximal rank as well. This implies by \cite[Lemma 18.52]{Milne} that $N_{G}(K^{\circ})^\circ=K^\circ$. By construction, $K$ is normal in $H$, thus $H$ is contained in $N_{G}(K)$, which in turn is contained in $N_{G}(K^{\circ})$. This implies that $K^\circ=H^\circ$, thus that $H/K$ is a finite group scheme, contradicting the fact that $H/K\simeq \Ga$.
	\end{proof}

	\begin{theo}
		\label{rank:t}
		Let $\calE_0^{\dagger}$ be an overconvergent $F$-isocrystal over $X_0$. The subgroup $G(\calE)\subseteq G(\calE^\dagger)$ has maximal rank. 
		
	\end{theo}
	\proof
	If we replace $\calE_0^{\dagger}$ with its semi-simplification with respect to a Jordan--H\"older filtration, we do not change the reductive rank of $G(\mathcal E^{\dagger})$ and $G(\mathcal E)$. Thus we may and do assume that $\calE_0^{\dagger}$ is semi-simple. This implies that $\mathcal E^{\dagger}$ is semi-simple as well.
By \cite[Corollary 3.4.2]{Dad}, we can find an overconvergent $F$-isocrystal $\calF_0^\dagger$ over $X_0$ which is a direct sum of irreducible overconvergent $F$-isocrystals with finite order determinant and such that $\calE^\dagger\simeq \calF^\dagger$. Thanks to [\textit{ibid.}, Theorem 3.6.6], we know that $\calF_0^{\dagger}$ is algebraic. Therefore, up to replacing $\calE_0^\dagger$ with $\calF_0^\dagger$, we may assume that $\calE_0^\dagger$ is algebraic.
	
	 Let $V_0^\dagger$ (resp. $V_0$) be the representation of $G(\calE^{\dagger}_0)$ (resp. $G(\calE_0)$) associated to $\calE_0^{\dagger}$ (resp. $\calE_0$). Choose a set of generators $\chi_{1,0},\dots,\chi_{n,0}$ of $X^*(G(\calE_0)^{\cst})$. These correspond to constant $F$-isocrystals over $X_0$ of rank $1$. Since every constant $F$-isocrystal comes from an overconvergent $F$-isocrystal, every $\chi_{i,0}$ extends to a character $\chi_{i,0}^{\dagger}$ of $\pi_1^{\Foi}(X_0)$, the Tannakian group of $\Foi(X_0)$. Note that, a priori, these characters do not factor through $G(\calE^\dagger_0)^{\mathrm{cst}}$ (see Remark \ref{dont-factor:r}). Take $$\widetilde{V}^{\dagger}_0:=V_0^{\dagger}\oplus \bigoplus_{i=1}^n \chi_{i,0}^{\dagger}$$ and write $\widetilde{V}_0$ for the induced representation of $\pi_1^{\Fisoc}(X_0)$, the Tannakian group of $\Fisoc(X_0)$. By construction, the groups of constant characters $X^*(G(\widetilde{V}_0)^{\cst})$ and $X^*(G(V_0)^{\cst})$ are canonically isomorphic and generated by the $\chi_{i,0}$. Moreover, since $\widetilde{V}^{\dagger}\simeq V^{\dagger}\oplus \Qpbar^{\oplus n}$ and  $\widetilde{V}\simeq V\oplus \Qpbar^{\oplus n}$, we get isomorphisms $G(\widetilde{V}^{\dagger})\simeq G(\calE^{\dagger})$ and $G(\widetilde{V})\simeq G(\calE)$. Therefore, it is enough to show that $\rk(G(\widetilde{V}^{\dagger}))=\rk(G(\widetilde{V}))$. 
	
	By Proposition \ref{conv-oconv-mon:p}, there exists a commutative diagram with exact rows
	\begin{center}
		\begin{tikzcd}
			0\arrow{r} & G(\widetilde{V})\arrow{r}\arrow[hook,d] & G(\widetilde{V}_0)\arrow{r}\arrow[hook,d] & G(\widetilde{V}_0)^{\cst}\arrow{r}\arrow[d, two heads] &0\\
			0\arrow{r} & G(\widetilde{V}^{\dagger})\arrow{r} & G(\widetilde{V}_0^{\dagger})\arrow{r} & G(\widetilde{V}_0^{\dagger})^{\cst}\arrow{r} & 0
		\end{tikzcd}
	\end{center}
where the first two vertical arrows are injective and the last one is surjective. As $\widetilde{V}_0^{\dagger}$ is still algebraic, by Corollary \ref{arit-rank:c}, $\rk(G(\widetilde{V}_0))=\rk(G(\widetilde{V}^{\dagger}_0))$.
	Since the reductive rank is additive in exact sequences, it is enough to show that $G(\widetilde{V}_0)^{\cst}$ and $G(\widetilde{V}^{\dagger}_0)^{\cst}$ have the same reductive rank. We will show that the morphism $\varphi: G(\widetilde{V}_0)^{\cst}\to G(\widetilde{V}_0^\dagger)^{\cst}$ of the previous diagram is actually an isomorphism. We already know that $\varphi$ is surjective. As $G(\widetilde{V}_0)^{\cst}$ and $G(\widetilde{V}_0^\dagger)^{\cst}$ are groups of multiplicative type by Corollary \ref{mult-type:c}, it remains to show that the map $\varphi^*:X^*(G(\widetilde{V}^{\dagger}_0)^{\cst})\to X^*(G(\widetilde{V}_0)^{\cst})$ is surjective. This is a consequence of the construction of $\widetilde{V}^{\dagger}_0$. Indeed, $X^*(G(\widetilde{V}_0)^{\cst})=X^*(G(V_0)^{\cst})$ is generated by $\chi_{1,0},\dots,\chi_{n,0}$ and for every $i$, the character $\chi_{i,0}^\dagger\in X^*(G(\widetilde{V}^{\dagger}_0)^{\cst})$ is sent by $\varphi^*$ to $\chi_{i,0}$. 
	\endproof
	\begin{coro}\label{red-rank-cst:c}
		Let $\calE_0^{\dagger}$ be an algebraic overconvergent $F$-isocrystal. The reductive rank of $G(\calE^{\dagger}_0)^{\cst}$ is the same as the one of $G(\calE_0)^{\cst}$.
	\end{coro}
	\begin{proof}
		The result follows from Corollary \ref{arit-rank:c} and Theorem \ref{rank:t}, thanks to Proposition \ref{conv-oconv-mon:p} and the additivity of the reductive ranks with respect to exact sequences.
	\end{proof}
	\begin{rema}\label{sl3:r}Using Theorem \ref{kedlaya:t}, one can show that when $\calE^\dagger$ is semi-simple, the functor $\langle\calE^\dagger\rangle\to \langle\calE\rangle$ is fully faithful. Therefore, in this case, $G(\mathcal E)\subseteq G(\mathcal E^{\dagger})$ is an epimorphic subgroup (cf. Remark \ref{epimorphic:r}). Nevertheless, Theorem \ref{rank:t} does not follow directly from this, because epimorphic subgroups can have, in general, lower reductive rank. For example, let $\KK$ be any field and let $G$ be the algebraic group $\mathrm{SL}_{3,\KK}$. The subgroup $H$ of $G$ defined by the matrices of the form
		\begin{center}
			
			$\begin{pmatrix}
			a      & 0 & * \\
			0 & a & * \\
			0  & 0& a^{-2}
			\end{pmatrix},$
		\end{center}
		with $a\in \KK^{\times}$, is the radical of a maximal parabolic subgroup of $G$. Therefore, by \cite[§2.(a) and §2.(d)]{Borel1}, $H$ is an epimorphic subgroup of $G$. On the other hand, the reductive rank of $H$ is $1$. Even more surprisingly, in characteristic $0$ \textit{every} almost simple group contains an epimorphic subgroup of dimension $3$ [\textit{ibid.}, §5.(b)].
	\end{rema}
	Another consequence of Theorem \ref{rank:t} is the following result that we will not use, but which has its own interest. We have already discussed it in §\ref{weak:ss}.
	\begin{coro}\label{weakweak:c}
		Let $\calE_0^\dagger$ be a overconvergent $F$-isocrystal and assume that $\calE^\dagger$ is semi-simple. Every $\calF_0\in \langle\calE_0 \rangle$ which is an extension of constant $F$-isocrystals is constant.
	\end{coro}
	\begin{proof}
		The statement is equivalent to the fact that the group $\Ext^1_{G(\calE)}(\Qpbar,\Qpbar)$ vanishes. The result then follows from Theorem \ref{rank:t} thanks to Lemma \ref{coh-van:l}.
	\end{proof}
	
	\section{A special case of a conjecture of Kedlaya}\label{KC:s}
	
	\subsection{Proof of the main theorem}
	As a consequence of the results of §\ref{max-tori:ss}, we obtain a special case of the conjecture in \cite[Remark 5.14]{Ked16}. We shall start with a finiteness result. We retain the notation as in §\ref{fin-fields:nota}.
	\begin{prop}
		\label{constant-objects:p}
		If $\calE^{\dagger}_0$ is an irreducible overconvergent $F$-isocrystal over $X_0$ with finite order determinant, then $G(\calE_0)^{\cst}$ is finite. In particular, every constant subquotient of the $F$-isocrystal $\calE_0$ is finite (cf. §\ref{tannakian:ss}).
	\end{prop}
	\proof
	We first notice that $\calE^\dagger_0$ is algebraic thanks to Deligne's conjecture, \cite[Lemma 4.1]{Abeesnault}. By Corollary \ref{mult-type:c}, we deduce that the algebraic groups $G(\calE_0)^\cst$ and $G(\calE_0^\dagger)^\cst$ are of multiplicative type and by Corollary \ref{red-rank-cst:c} that they have the same dimensions. Hence it is enough to show that $G(\calE_0^\dagger)^{\cst}$ is finite. Thanks to Proposition \ref{conv-oconv-mon:p}, this is the same as showing that $G(\calE^\dagger)$ is a finite index subgroup of $G(\calE_0^\dagger)$ or, equivalently, that $G(\calE^\dagger)^{\circ}=G(\calE_0^\dagger)^{\circ}$. By \cite[Corollary 3.4.5]{Dad}, $G(\calE^\dagger)^{\circ}$ is the derived subgroup of the reductive group $G(\calE_0^\dagger)^{\circ}$, so that it is enough to show that $G(\calE_0^\dagger)^{\circ}$ has finite centre. This follows from the fact that the natural faithful representation of $G(\calE_0^\dagger)$ is irreducible with finite order determinant.

	\endproof
	
	\begin{rema}\label{gmt:r}Apart from the known cases of Deligne's conjecture, Proposition \ref{constant-objects:p} ultimately relies on class field theory. Indeed, the key input for \cite[Corollary 3.4.5]{Dad} is the \textit{global monodromy theorem} for overconvergent $F$-isocrystals, proven by Crew (over curves) in \cite{CrewMon}. This theorem, as in the $\ell$-adic case, essentially follows from class field theory. 
		\end{rema}

	\begin{theo}\label{key:t}
		Let $\calE_0^{\dagger}$ be an irreducible overconvergent $F$-isocrystal over $X_0$. If $\calE_0$ admits a subobject of minimal slope $\calF_0\subseteq \calE_0$ with a non-zero morphism $\calF\to\mathbbm 1$, then $\calF=\calE$ and $\calE\simeq\mathbbm 1$.
	\end{theo}
	
	\begin{proof}
		Observe that both the hypothesis and the conclusion are invariant under twist. Thus, by \cite[Lemma 6.1]{Abe2}, we may assume that the determinant of $\calE_0^{\dagger}$ is of finite order, hence unit-root. We first prove that $\calE_0^{\dagger}$ is unit-root as well. If $r$ is the rank of $\calE_0^{\dagger}$, since 
		$$\sum_{i=1}^r a_i^{\eta}(\calE_0^\dagger)=a_1^\eta(\det(\calE_0^{\dagger}))=0 \quad \text{and}\quad a_1^{\eta}(\calE_0^\dagger)\leq \dots \leq a_r^{\eta}(\calE_0^\dagger),$$ it suffices to show that $a_1^{\eta}(\calE_0^\dagger)=0$. Let $\calF\twoheadrightarrow \calT$ be the maximal trivial quotient of $\calF$. Since  $\calF\twoheadrightarrow \calT$ is maximal, it is preserved by the action of $F$, hence it descends to a quotient $\calF_0\twoheadrightarrow \calT_0$, where $\calT_0$ is a constant $F$-isocrystal. The overconvergent $F$-isocrystal $\calE_0^{\dagger}$ satisfies the assumptions of Proposition \ref{constant-objects:p}, hence $G(\calT_0)$ is finite. As the $F$-isocrystal $\calF_0$ is isoclinic and it admits a non-zero quotient which is finite, it is unit-root. This implies that $a_1^{\eta}(\calE_0^{\dagger})=0$, as we wanted.

		We now prove that $\calE_0^\dagger$ has rank 1. Since $\mathcal E_0^{\dagger}$ is unit-root, as a consequence of a theorem of Tsuzuki, \cite[Theorem 3.9]{Ked16}, the functor $\langle \calE_0^\dagger \rangle\to \langle \calE_0 \rangle$ is an equivalence of categories. Therefore, if $\calE_0$ has a constant subquotient, the same is true for $\calE_0^{\dagger}$. But $\calE_0^{\dagger}$ is irreducible by assumption, thus it has to be itself a constant $F$-isocrystal. By Remark \ref{cst:r},  $\calE_0^{\dagger}$ corresponds to an irreducible $\Qpbar$-linear representation of $\Z$. Therefore, since $\Qpbar$ is algebraically closed, $\calE_0^\dagger$ has rank 1, as we wanted.
	\end{proof}
	\begin{rema}\label{Relationships:r}
		The statement of Theorem \ref{key:t} is false in general if we do not assume that $\calF_0\subseteq \calE_0$ is of minimal slope. A counterexample was provided by the second named author, \cite[Example 5.15]{Ked16}. A modified version of this conjecture has been proven by Tsuzuki in \cite{Tsuzuki} (over finite field and for curves over perfect fields). In his proof he performs a geometric ($p$-adic) analysis on the behaviour of the $F$-isocrystals both at the level of the completion along closed points, for example [\textit{ibid.}, Theorem 2.14], and at a global level, [\textit{ibid.}, Theorem 3.27]. In our proof, instead, we exploit an “arithmetic rigidity” specific to the situation over finite fields. This rigidity is expressed thanks to \cite[Lemma 6.1]{Abe2}, which mainly uses class field theory, and the existence of Frobenius tori of maximal dimension, which cannot happen over more general perfect fields.
		
	\end{rema}
	\subsection{Some consequences}
	
	\begin{coro}\label{pure-geom-surj:c}
		Let $\calE_0^{\dagger}$ be an overconvergent $F$-isocrystal over $X_0$ and let $\calF_0$ be a subobject of $\calE_0$ of minimal generic slope. If $\mathcal E^{\dagger}$ is semi-simple, then the restriction morphism $\Hom(\calE, \mathbbm 1)\to \Hom(\calF, \mathbbm 1)$ is surjective.
	\end{coro}
	\begin{proof}
		As $\mathcal E^{\dagger}$ is semi-simple if we replace $\calE_0^{\dagger}$ with its semi-simplification with respect to a Jordan--H\"older filtration, we do not change the isomorphism class of $\calE^{\dagger}$. Thus we may and do assume that $\calE_0^{\dagger}$ is semi-simple.
		The proof is then an induction on the number $n$ of summands of any decomposition of $\calE_0^{\dagger}$ in irreducible overconvergent $F$-isocrystals. If $n=1$ this is an immediate consequence of Theorem \ref{key:t}. Suppose now that the result is known for every positive integer $m<n$. Take an irreducible subobject $\mathcal G_0^{\dagger}$ of $\mathcal E^{\dagger}_0$, write $\mathcal H_0:=\mathcal G_0\times_{\mathcal E_0} \mathcal F_0$ and consider the following commutative diagram with exact rows and injective vertical arrows
		\begin{center}
			\begin{tikzcd}
				0\arrow{r} & \mathcal H\arrow{r}\arrow[d,hook] & \mathcal F\arrow[d,hook]\arrow{r} & \mathcal F/\mathcal H\arrow[d,hook]\arrow{r} & 0\\
				0\arrow{r} & \mathcal G \arrow{r} & \mathcal E \arrow{r} & \mathcal{E}/\mathcal G\arrow{r} & 0.
			\end{tikzcd}
		\end{center}
		
		As $\calE_0^{\dagger}$ is semi-simple, the quotient $\calE_0^{\dagger}\twoheadrightarrow\calE_0^{\dagger}/\calG_0^{\dagger}$ admits a splitting. This implies that the lower exact sequence splits. We apply the functor $\Hom(-,\mathbbm 1)$ and we get the following commutative diagram with exact rows
		\begin{center}
			\begin{tikzcd}
				0\arrow{r} & \Hom(\mathcal E/\mathcal G,\mathbbm 1)\arrow{r}\arrow{d} & \Hom(\mathcal E,\mathbbm 1)\arrow{d}\arrow{r} & \Hom(\mathcal G,\mathbbm 1)\arrow{d}\arrow{r} & 0\\
				0\arrow{r} & \Hom(\mathcal F/\mathcal H,\mathbbm 1) \arrow{r} & \Hom(\mathcal F,\mathbbm 1) \arrow{r} & \Hom(\mathcal H,\mathbbm 1).
			\end{tikzcd}
		\end{center}
		Since $\mathcal H_0$ and $\mathcal F_0/\mathcal H_0$ are subobjects of minimal slope of $\mathcal G_0$ and $\mathcal E_0/\mathcal G_0$ respectively, by the inductive hypothesis the left and the right vertical arrows are surjective. By diagram chasing, this implies that the central vertical arrow is also surjective, as we wanted.
	\end{proof}
	
	\begin{rema}By the theory of weights, if $\calE_0^\dagger$ is pure then $\calE^\dagger$ is semi-simple. Hence one can apply Corollary \ref{pure-geom-surj:c} in this situation. The theorem is false instead without the assumption that $\mathcal E^{\dagger}$ is semi-simple. For example, when $X_0=\mathbb{G}_{m,\Fq}$, there exits a non-trivial extension $$0\to\mathbbm 1^{\dagger}_0\to \calE_0^{\dagger}\to (\mathbbm 1^{\dagger}_0)^{(q)}\to 0,$$ which does not split in $\oi(X_0)$. If $\calF_0 \subseteq \calE_0$ is the rank $1$ trivial subobject of $\calE_0$, then the map $\Hom(\calE, \mathbbm 1)\to \Hom(\calF, \mathbbm 1)$ is the zero map, even though $\Hom(\calF, \mathbbm 1)=\Qpbar$.
	\end{rema}
	We end the section presenting a variant of Corollary \ref{pure-geom-surj:c}, where we rather consider morphisms in $\Fisoc(X_0)$.
	\begin{coro}
		Let $\calE_0^{\dagger}$ be an algebraic semi-simple overconvergent $F$-isocrystal with constant Newton polygons and with minimal slope equal to $0$. Let $\calE^1_0\subseteq \mathcal E_0$ be the maximal unit-root subobject of $\mathcal E_0$. The restriction morphism $\Hom(\calE_0, \mathbbm 1_0)\to \Hom(\calE^1_0, \mathbbm 1_0)$ is an isomorphism.
	\end{coro}
	\begin{proof}
		Since $\mathcal E^{\dagger}_0$ is semi-simple, the overconvergent isocrystal $\mathcal E^{\dagger}$ is semi-simple as well. By Corollary \ref{pure-geom-surj:c}, the restriction morphism $\Hom(\calE, \mathbbm 1)\to \Hom(\calE^1, \mathbbm 1)$ is then surjective. The vector space $\Hom(\calE, \mathbbm 1)$ endowed with its Frobenius action corresponds, via the equivalence in Remark \ref{cst:r}, to the maximal constant subobject of $\calE_{0}^\vee$. Since constant $F$-isocrystals are overconvergent, the latter is the same as the maximal constant subobject of $(\mathcal E^{\dagger}_0)^\vee$. Combining this with the semi-simplicity assumption on $\mathcal E^{\dagger}_0$, we deduce that the action of $F$ on $\Hom(\calE, \mathbbm 1)$ is diagonalizable. Therefore, the restriction morphism
		$$\Hom(\calE_0, \mathbbm 1_0)=\Hom(\calE, \mathbbm 1)^{F}\to \Hom(\calE^1, \mathbbm 1)^{F}=\Hom(\calE^1_0, \mathbbm 1_0)$$ is surjective as well. To prove that it is injective, we note that a morphism from $\calE_0$ to $\mathbbm 1_0$ which factors through $\mathcal E_0/\mathcal E^1_0$ is the zero morphism because $\mathcal E_0/\mathcal E^1_0$ has positive slopes while $\mathbbm 1_0$ has slope $0$. This concludes the proof.
	\end{proof}

	\section{An extension of the theorem of Lang--Néron}
	\label{perfect-p-tor:s}
	
	\subsection{\texorpdfstring{$p$}{}-torsion and \texorpdfstring{$p$}{}-divisible groups}
	We exploit here Corollary \ref{pure-geom-surj:c} to prove the following result on the torsion points of abelian varieties. Let $\F\subseteq k$ be a finitely generated field extension and let $k^{\perf}$ be a perfect closure of $k$. Recall that for an abelian variety $A$ over $k$ we write $\Tr_{k/\F}(A)$ for its $k/\F$-trace (cf. \cite[§6]{Conrad}).
	\begin{theo}\label{perfect-p-tor:t}
			If $A$ is an abelian variety over $k$ such that $\Tr_{k/\F}(A)=0$, then the group $A(k^{\mathrm{perf}})_{\mathrm{tors}}$ is finite.
	\end{theo}
	As we have already discussed in Remark \ref{intro-p-tor:r}, thanks to Theorem \ref{Lang-Neron:t} it is enough to show that $A[p^\infty](k^{\perf})$ is finite. 
	
	\begin{nota}\label{model:ss}
	We fix some notation for §\ref{perfect-p-tor:s}. Let $A$ be an abelian variety over $k$ (we do not ask that $\Tr_{k/\F}(A)=0$). Since $A/k$ is of finite type over $k$, we can choose a subfield $k_0\subseteq k$ which is finitely generated over $\F_p$ such that $k=\F k_0$ and such that there exists an abelian variety $A_0/k_0$ endowed with an isomorphism $A\simeq A_0\otimes_{k_0}k$. Let $\Fq$ be the algebraic closure of $\F_p$ in $k_0$. By spreading out, we can choose a smooth geometrically connected variety $X_0$ over $\Fq$ with $\Fq(X_0)\simeq k_0$ and an abelian scheme $f_0:\frakA_0\to X_0$ endowed with an isomorphism $\mathcal A\times_X \Fq(X_0)\simeq A_0$. Since the Newton polygons of  $f_0:\frakA_0\to X_0$ are constant on a dense open subset of $X_0$,after shrinking $X_0$ we may assume that Newton polygons $f_0:\frakA_0\to X_0$ are constant. We denote by $X$ and $\frakA$ the extension of scalars of $X_0$ and $\frakA_0$ from $\Fq$ to $\F$.
	\end{nota}

	\begin{lemm}\label{infinity:l}
		Let $A$ be an abelian variety over $k$ and let $\frakA_0/X_0$ be a model of $A$ as in §\ref{model:ss}. If $|A[p^{\infty}](k^{\mathrm{perf}})|=\infty$, then there exists an injective morphism $(\Qp/\Zp)_X\hookrightarrow\frakA[p^\infty]^{\mathrm{\acute{e}t}}$.
	\end{lemm}
	\begin{proof}
		We first prove that the group $A[p^{\infty}](k^{\mathrm{perf}})$ is isomorphic to $\frakA[p^\infty]^{\et}(X)$, showing thereby that $\frakA[p^\infty]^{\et}(X)$ is infinite as well.
		As $k^{\mathrm{perf}}$ is a perfect field, the map $$A[p^{\infty}](k^{\mathrm{perf}})\to A[p^{\infty}]^{\et}(k^{\mathrm{perf}})=A[p^{\infty}]^{\et}(k),$$ induced by the quotient $A[p^\infty]\twoheadrightarrow A[p^\infty]^\et$ is an isomorphism. In addition, since the Newton polygons of $f:\frakA\to X$ are constant, the $p$-ranks of the fibres of $f$ are all equal. Therefore, for every $i$, the constructible \'etale sheaf $\frakA[p^i]^{\et}$ over $X$ is locally constant. Since $X$ is normal, the restriction morphism $\frakA[p^\infty]^{\et}(X)\to A[p^\infty]^{\et}(k)$ induced by the inclusion of the generic point $\Spec(k)\hookrightarrow X$ is then an isomorphism. These two observations show that $A[p^{\infty}](k^{\mathrm{perf}})\simeq\frakA[p^\infty]^{\et}(X)$, as we wanted.
		
		Since $|\frakA[p^\infty]^{\et}(X)|=|A[p^{\infty}](k^{\mathrm{perf}})|=\infty$, a standard compactness argument shows that there exists a morphism $(\Qp/\Zp)_X\hookrightarrow\frakA[p^\infty]^{\et}$. For the reader convenience, we quickly recall it. We define a partition of $\frakA[p^\infty]^{\et}(X)$ in subsets $\{\Delta_i\}_{i\in \N}$ in the following way. Let  $\Delta_0:=\{0\}$ and for $i>0$, let $\Delta_i:=\frakA[p^i]^{\et}(X)\setminus \frakA[p^{i-1}]^{\et}(X)$. When $j\geq i$, the  multiplication by $p^{j-i}$ induces a map $\Delta_j\to \Delta_i$. These maps make $\{\Delta_i\}_{i\in \N}$ an inverse system. We claim that every $\Delta_i$ is non-empty. Suppose by contradiction that for some $N\in \N$, the set $\Delta_N$ is empty. By construction, for every $i\geq N$ the sets $\Delta_i$ are empty as well. Since every $\Delta_i$ is finite, this would imply that $\frakA[p^\infty]^{\et}(X)$ is also finite, which is a contradiction. As every $\Delta_i$ is non-empty, by Tychonoff's theorem, the projective limit $\varprojlim \Delta_i$ is non-empty. The choice of an element $(P_i)_{i\in \N}\in \varprojlim \Delta_i$ induces an injective map  $(\Qp/\Zp)_X\hookrightarrow\frakA[p^\infty]^{\et}$, given by the assignment $[1/p^i]\mapsto P_i$. This yields the desired result.
	\end{proof}
	
	\subsection{Reformulation with the crystalline Dieudonné theory}
	We restate the classical crystalline Dieudonné theory in our setup.
	\subsubsection{}
	
	\label{cdt:ss}
	
	Let $$\mathbb{D}:\{\textrm{$p$-divisible groups $/X$}\}\to \Fcrys(X/W(\F))$$ be the \textit{crystalline Dieudonné module (contravariant) functor}, where $\Fcrys(X/W(\F))$ is the category of coherent $\Zp$-linear $F$-crystals (cf. \cite[Définition 3.3.6]{BBM82}). By \cite[Main Theorem 1]{deJ95} the functor is fully faithful and, by \cite[Théorème 2.5.6.(ii)]{BBM82}, there is a canonical isomorphism $\mathbb{D}(\frakA[p^\infty])\simeq R^1f_{\crys*}\calO_{\frakA}$. 
	Extending the scalars to $\Qpbar$ and post-composing with the functor of Theorem \ref{Ogus:t}, we define a $\Qpbar$-linear fully faithful contravariant functor $$\mathbb{D}_{\Qpbar}:\{\textrm{$p$-divisible groups $/X$}\}_{\Qpbar}\to \Fisoc(X).$$
	The functor sends the trivial $p$-divisible group $(\Qp/\Zp)_{X,\Qpbar}$ to the ($\Qpbar$-linear) $F$-isocrystal $(\calO_X,\id)$ over $X$. By the classification of $p$-divisible groups in terms of Dieudonn\'e modules over perfect fields, if $X=\Spec(\F)$ the functor induces an equivalence $$\mathbb{D}_{\Qpbar}:\{\textrm{$p$-divisible groups $/\Spec(\F)$}\}_{\Qpbar}\iso \Fisoc_{[0,1]}(\Spec(\F)),$$
	where $\Fisoc_{[0,1]}(\Spec(\F))$ is the category of $F$-isocrystals with slopes between $0$ and $1$. Since $\mathbb{D}_{\Qpbar}$ is compatible with base change, this implies that for every $X$, the functor $\mathbb{D}_{\Qpbar}$ is exact, it preserves the heights/ranks and it sends étale $p$-divisible groups to unit-root $F$-isocrystals. 
\subsubsection{}\label{Ete:ss}
By \cite[Théorème 7]{Ete02}, the $F$-isocrystal $R^1f_{0,\crys*}\calO_{\frakA_0}$ over $X_0$ comes from an overconvergent $F$-isocrystal, which we denote by $\calE_0^\dagger$. This overconvergent $F$-isocrystal is semi-simple, as proven over curves in \cite[Theorem 1.2]{Pal15} and explained in general in the proof of \cite[Theorem 5.1.6]{D'Ad20}. Let $\calF_0$ be the maximal unit-root subobject of $\calE_0$ and let $(\calF_0)_X$ be the inverse image of $\calF_0$ to $X$, as an $F$-isocrystal. As usual, we also denote by $\calE$ and $\calF$ the isocrystals over $X_0$ obtained from $\calE_0$ and $\calF_0$ by forgetting the Frobenius structure.

	 By the discussion in §\ref{cdt:ss}, we have the following result.
	\begin{lemm}\label{identification:l}Let $A$ be an abelian variety over $k$ and let $\frakA_0/X_0$ be a model of $A$ as in §\ref{model:ss}. The quotient $\frakA[p^\infty]\twoheadrightarrow \frakA[p^\infty]^{\et}$ is sent by $\mathbb{D}_{\Qpbar}$ to the natural inclusion $(\calF_0)_X\hookrightarrow (\calE_0)_X.$
	\end{lemm}
By using Lemma \ref{identification:l}, we can reformulate Lemma \ref{infinity:l} in the language of $F$-isocrystals.
	\begin{coro}\label{infty-isoc:c}
		If $|A[p^{\infty}](k^{\mathrm{perf}})|=\infty$, then there exists a quotient  $(\calF_0)_X\twoheadrightarrow (\calO_X,\id)$.
	\end{coro}
	
	\begin{proof}
	Thanks to Lemma \ref{infinity:l}, if $|A[p^{\infty}](k^{\mathrm{perf}})|=\infty$, then there exists an injective morphism $(\Qp/\Zp)_X\hookrightarrow\frakA[p^\infty]^{\et}$.
		By Lemma \ref{identification:l}, after we extend the scalars to $\Qpbar$, this morphism is sent by $\mathbb{D}_{\Qpbar}$ to a quotient $(\calF_0)_X\twoheadrightarrow (\calO_X,\id)$.
	\end{proof}    
	\subsection{End of the proof}
	We need to rephrase the finiteness of torsion points given by the theorem of Lang--Néron in terms of morphisms of isocrystals on $X_0$. This will lead finally to the proof of Theorem \ref{perfect-p-tor:t}. 
	\begin{prop}\label{trace-p-adic:p} Let $A$ be an abelian variety over $k$, let $\frakA_0/X_0$ be a model of $A$ as in §\ref{model:ss}, and let $\calE_0$ and $\calF_0$ be the $F$-isocrystals defined in §\ref{Ete:ss}. If for some $n>0$, there exists a morphism $\calE\to\calO_{X_0}^{\oplus n}$ such that the image of $\calF$ is non-zero, then $\Tr_{k/\F}(A)\neq 0$.
	\end{prop}
	\begin{proof}
		Write $\calE \twoheadrightarrow \calT$ for the maximal trivial quotient of $\mathcal E$. Since $\calE\twoheadrightarrow \calT$ is maximal, it is preserved by the action of $F$ on $\calE$, hence it defines a quotient $\calE_0\twoheadrightarrow \calT_0$, where $\calT_0$ is the maximal constant quotient of $\calE_0$. We base change this quotient from $X_0$ to $X$, as a morphism of $F$-isocrystals, obtaining a quotient $(\calE_0)_X\twoheadrightarrow (\calT_0)_X$ in $\Fisoc(X)$. Since $\calT_0$ is an $F$-isocrystal coming from $\Spec(\Fq)$, the $F$-isocrystal $(\calT_0)_X$ comes from $\Spec(\F)$. Thanks to the Dieudonn\'e--Manin decomposition, $(\calT_0)_X$ decomposes in $\Fisoc(X)$ as $$(\calT_0)_X=(\calT_0')_X\oplus (\calO_X^{\oplus m},\id)$$ where $(\calO_X^{\oplus m},\id)$ is the maximal unit-root subobject of $(\calT_0)_X$ and $m\geq 0$. As $\calF_0$ is unit-root, it is sent via the quotient $(\calE_0)_X\twoheadrightarrow (\calT_0)_X$ to a unit-root $F$-subisocrystal of $(\calT_0)_X$. By the assumption, the image of $(\calF_0)_X$ is non-zero, so that $m$ has to be greater than $0$. Thus $(\calE_0)_X$ admits a surjective morphism to $(\calO_X,\id)$ in $\Fisoc(X)$. Since $\mathbb{D}_{\Qpbar}$ is fully faithful, such a quotient comes from a monomorphism $(\Qp/\Zp)_{X,\Qpbar}\hookrightarrow\frakA[p^\infty]_{\Qpbar}$ in the category of $p$-divisible groups with coefficients in $\Qpbar$. After multiplying by some element in $\overline \Q_p^{\times}$, the map defines an injection $(\Qp/\Zp)_{X}\hookrightarrow\frakA[p^\infty]$ of $p$-divisible groups over $X$. By Theorem \ref{Lang-Neron:t}, this implies that $\Tr_{k/\F}(A)\neq 0$.
	\end{proof}
	
	\begin{proof}[Proof of Theorem \ref{perfect-p-tor:t}]
		Assume by contradiction that $|A[p^{\infty}](k^{\mathrm{perf}})|=\infty$. By Corollary \ref{infty-isoc:c}, we have a quotient $(\calF_0)_X\twoheadrightarrow (\calO_X,\id)$ in $\Fisoc(X)$. Forgetting the Frobenius structure we get a quotient $\calF_X\twoheadrightarrow \calO_X$ in $\Isoc(X)$. Let $\calF_X\twoheadrightarrow \calO_{X}^{\oplus n}$ be the maximal trivial quotient of $\calF_X$, where $n>0$. By maximality, the action of the absolute Galois group of $\Fq$ on $\calF_X$ preserves this quotient. Therefore, the morphism $\calF_X\twoheadrightarrow \calO_X^{\oplus n}$ descends to a quotient $\calF\twoheadrightarrow\calT$ in $\Isoc(X_0)$. Since the descent datum defining $\calT$ is the pullback of a descent datum over $\Spec(\F)$, the convergent isocrystal $\calT$ is isomorphic to the pullback of a convergent isocrystal over $\Fq$. Notice that every convergent isocrystal over $\Spec(\Fq)$ is trivial, so that $\calT\simeq \calO_{X_0}^{\oplus n}$. Now, thanks to the fact that $\calE^\dagger_0$ is semi-simple, as explained in §\ref{Ete:ss}, we can use Corollary \ref{pure-geom-surj:c} and extend the quotient $\calF\twoheadrightarrow \calO_{X_0}^{\oplus n}$ to a quotient $\calE\twoheadrightarrow\calO_{X_0}^{\oplus n}$ in $\Isoc(X_0)$. By Proposition \ref{trace-p-adic:p}, this would imply that $\mathrm{Tr}_{k/\mathbb{F}}(A)\neq 0$, which leads to a contradiction.
	\end{proof}
	
	\begin{rema}\label{Langlands:r}The proofs of Theorem \ref{frob-tori:t} and Proposition \ref{constant-objects:p} rely on the known cases of Deligne's conjecture. In particular, they rely on the Langlands correspondence for lisse sheaves proven in \cite{Laf} and the Langlands correspondence for overconvergent $F$-isocrystals proven in \cite{Abe}. We want to point out that to prove Theorem \ref{perfect-p-tor:t} we do not need this theory.
	More precisely, when $\calE_0^\dagger$ is an overconvergent $F$-isocrystal which comes from geometry (cf. §\ref{weak:ss}), Theorem \ref{frob-tori:t} can be proven more directly, as explained in \cite[Remark 4.2.13]{Dad}. Even in the proof of Proposition \ref{constant-objects:p}, if $\calE_0^\dagger$ comes from geometry we do not need [\textit{ibid.}, Theorem 3.6.6]. This applies, for example, to the overconvergent $F$-isocrystals appearing in §\ref{perfect-p-tor:s}.	To summarize, in the proof of Theorem \ref{perfect-p-tor:t}, the main ingredients are: class field theory (see Remark \ref{gmt:r}), the theory of weights for overconvergent $F$-isocrystals, and Kedlaya's full faithfulness theorem.
	\end{rema}

	\bibliographystyle{ams-alpha}

\end{document}